\documentclass[11pt]{article}
\title{Roots of unity and torsion points of abelian varieties}
\usepackage{amsfonts}
\usepackage{amsmath}
\usepackage{amssymb}
\usepackage{amsthm}
\usepackage{a4wide}
\usepackage{todonotes}
\usepackage[utf8]{inputenc}

\newtheorem{theorem}{Theorem}
\numberwithin{theorem}{section}

\newtheorem{lemma}[theorem]{Lemma}

\newtheorem{proposition}[theorem]{Proposition}

\theoremstyle{definition}
\newtheorem{remark}[theorem]{Remark}
\newtheorem{example}[theorem]{Example}
\newtheorem{definition}[theorem]{Definition}

\newcommand{\simb}{\square}
\newcommand{\SpecialFiber}[1]{#1_\ell(\ell)}
\newcommand{\abGal}[1]{\operatorname{Gal}(\overline{#1}/#1)}
\newcommand{\MT}{\operatorname{MT}}
\newcommand{\Hg}{\operatorname{Hg}}
\newcommand{\Smooth}{\mathcal{T}^0}
\newcommand{\StabilizerAlg}{\mathcal{T}}
\newcommand{\Stabilizer}{T}
\newcommand{\OtherStabilizer}{U}
\newcommand{\uguale}{\circeq}

\usepackage[british]{babel}

\usepackage[style]{abstract}
\usepackage[affil-it]{authblk}

\date{}
\begin{document}

\renewcommand{\abstitlestyle}[1]{\center{\small{\textsc{#1}}}}

\author{Davide Lombardo
  \thanks{\texttt{davide.lombardo@math.u-psud.fr}}}
\affil{Département de Mathématiques d'Orsay}
\maketitle

\vspace{-30pt}
\begin{abstract}
We answer a question raised by Hindry and Ratazzi concerning the intersection between cyclotomic extensions of a number field $K$ and extensions of $K$ generated by torsion points of an abelian variety over $K$. 
We prove that the property called $(\mu)$ in \cite{MR2643390} holds for any abelian variety, while the same is not true for the stronger version of the property introduced in \cite{MR2862374}.

\medskip

\end{abstract}


\noindent \textbf{MSC classes:} 11J95, 11G10, 14K15

\maketitle

\section{Introduction}
In this paper we consider the following problem: given a number field $K$, an abelian variety $A/K$ (of dimension $g$), a prime $\ell$, and a finite subgroup $H$ of $A[\ell^\infty]$, how does the number field $K(H)$ intersect the $\ell$-cyclotomic extension $K(\mu_{\ell^\infty})$? More precisely, is the intersection completely accounted for by the fact that $K(H)$ contains the image of the Weil pairing $H \times H \to \mu_{\ell^\infty}$?
In order to study this question, Hindry and Ratazzi have introduced in \cite{MR2643390} and \cite{MR2862374} two variants of a property they call $(\mu)$, and which we now recall.
We fix a polarization $\varphi:A \to A^\vee$ and, for every $n \geq 0$, we denote by $e_{\ell^n}$ the $\ell^n$-Weil pairing $A[\ell^n] \times A[\ell^n] \to \mu_{\ell^n}$ given by composing the usual Weil pairing $A[\ell^n] \times A^\vee[\ell^n] \to \mu_{\ell^n}$ with the map $A[\ell^n] \to A^\vee[\ell^n]$ induced by $\varphi$. If $H$ is a finite subgroup of $A[\ell^\infty]$ we now set
\[m_1(H)=\max\left\{k\in\mathbb{N}\ |\ \exists n\geq 0, \ \exists P,Q \in H\text{ of order }\ell^n\text{ such that } e_{\ell^n}(P,Q) \text{ generates }\mu_{\ell^k}\right\}.\]

Following \cite{MR2862374} we can then introduce the following definition: 

\begin{definition}\label{def_ms}
We say that $(A/K, \varphi)$ satisfies property $(\mu)_s$ (where ``s" stands for ``strong") if there exists a constant $C>0$, depending on $A/K$ and $\varphi$, such that for all primes $\ell$ and all finite subgroups $H$ of $A[\ell^\infty]$ the following inequalities hold:
\[
\frac{1}{C} [K(\mu_{\ell^{m_1(H)}}):K] \leq [K(H)\cap K(\mu_{\ell^{\infty}}) : K] \leq C [K(\mu_{\ell^{m_1(H)}}):K].
\]
\end{definition}

\begin{remark}
It is easy to see that the choice of the polarization $\varphi$ plays essentially no role, and $(A/K,\varphi)$ satisfies property $(\mu)_s$ for a given $\varphi$ if and only $(A/K,\psi)$ satisfies property $(\mu)_s$ for every polarization $\psi$ of $A/K$ (possibly for different values of the constant $C$); for this reason we shall simply say that $A/K$ satisfies property $(\mu)_s$ when it does for one (hence any) polarization. It is shown in \cite{MR2862374} that if $A/K$ satisfies the Mumford-Tate conjecture and has Mumford-Tate group isomorphic to $\operatorname{GSp}_{2\dim A, \mathbb{Q}}$, then property $(\mu)_s$ holds for $A$.
\end{remark}
We also consider the following variant of property $(\mu)_s$, which we call $(\mu)_w$ (``weak"), and which was first introduced in \cite[Définition 6.3]{MR2643390}:
\begin{definition}
We say that $A$ satisfies property $(\mu)_w$ if the following is true: there exists a constant $C>0$, depending on $A/K$, such that for all primes $\ell$ and all finite subgroups $H$ of $A[\ell^\infty]$ there exists $n \in \mathbb{N}$ (in general depending on $\ell$ and $H$) such that
\begin{equation}\label{eq_WeakMu1}
\frac{1}{C} \left[ K(\mu_{\ell^n}) : K \right] \leq \left[ K(H) \cap K(\mu_{\ell^\infty}) : K \right] \leq C \left[ K(\mu_{\ell^n}) : K \right].
\end{equation}
\end{definition}
Clearly, property $(\mu)_s$ implies property $(\mu)_w$. In this paper we show the following two results:
\begin{theorem}\label{thm_mw}
Let $K$ be a number field and $A/K$ be an abelian variety. Property $(\mu)_w$ holds for $A$.
\end{theorem}

\begin{theorem}\label{thm_ms}
There exists an abelian fourfold $A$, defined over a number field $K$, such that $\operatorname{End}_{\overline{K}}(A)=\mathbb{Z}$ and for which property $(\mu)_s$ does not hold. More precisely, such an $A$ can be taken to be any member of the family constructed by Mumford in \cite{MR0248146}.
\end{theorem}

The most surprising feature of the counterexample given by theorem \ref{thm_ms} is the condition $\operatorname{End}_{\overline{K}}(A)=\mathbb{Z}$. Indeed, one is easily led to suspect that the possible failure of property $(\mu)_s$ is tied to the presence of additional endomorphisms, as the following two examples show; theorem \ref{thm_ms}, however, demonstrates that $(\mu)_s$ can fail even in the favorable situation when $A$ has no extra endomorphisms. Notice however that an $A$ as in theorem \ref{thm_ms} has the property that $A^2$ supports ``exceptional" Tate classes, cf.~\cite{Moonen95hodgeand}, so the failure of property $(\mu)_s$ in this case can be understood in terms of the existence of certain algebraic cycles in the cohomology of $A$ which do not correspond to endomorphisms.

\begin{example} Property $(\mu)_s$ does not hold for abelian varieties of CM type. Indeed, let $A/K$ be an abelian variety of dimension $g$ admitting complex multiplication (over $K$) by an order $R$ in the ring of integers of the CM field $E$. Let $\ell$ be a prime that splits completely in $E$ and does not divide the index $[\mathcal{O}_E : R]$: we then have $R \otimes \mathbb{Z}_\ell \cong \mathcal{O}_E \otimes \mathbb{Z}_\ell \cong \mathbb{Z}_\ell^{2g}$, and by the theory of complex multiplication the action of $\operatorname{Gal}\left(\overline{K}/K \right)$ on $T_\ell(A)$ factors through $(R \otimes \mathbb{Z}_\ell)^\times$. It follows that in suitable coordinates the action of $\operatorname{Gal}\left(\overline{K}/K \right)$ on $A[\ell^n]$ is through diagonal matrices in $\operatorname{GL}_{2g}(\mathbb{Z}/\ell^n\mathbb{Z})$. Let now $P$ be the $\ell^n$-torsion point of $A$ which, in these coordinates, is represented by the vector $(1,\ldots,1)$. By our choice of coordinates, the Galois group of $K(A[\ell^n])$ over $K(P)$ is contained in
\[
\left\{ \sigma =\left(\begin{matrix} \sigma_{1,1} \\ & \sigma_{2,2} \\ & & \ddots \\ & & & \sigma_{2g,2g} \end{matrix} \right) \in \operatorname{GL}_{2g}\left(\mathbb{Z}/\ell^n\mathbb{Z} \right) \bigm\vert \sigma \cdot \left( \begin{matrix} 1 \\ \vdots \\ 1 \end{matrix}\right)=\left( \begin{matrix} 1 \\ \vdots \\ 1 \end{matrix}\right) \right\},
\]
a group which is clearly trivial: in other words, we have $K(P)=K(A[\ell^n])$. Let now $H$ be the group generated by $P$. It is clear that $m_1(H)=0$, because $H$ is cyclic, but on the other hand $K(H)=K(P)=K(A[\ell^n])$ contains a primitive $\ell^n$-th root of unity: since there are infinitely many primes $\ell$ satisfying our assumptions, this clearly contradicts property $(\mu)_s$ for $A$. In particular, this shows that in general property $(\mu)_s$ does not hold even if we restrict to the case of $H$ being cyclic.
\end{example}

\begin{example}\label{ex_SelfProducts}Property $(\mu)_s$ does not hold for self-products (this example has been pointed out to the author by Antonella Perucca). Let $B/K$ be any abelian variety and $P,Q$ be points of $B[\ell^n]$ such that $e_{\ell^n}(P,Q)$ generates $\mu_{\ell^n}$. Consider now $A=B^2$ and $H=\langle (P,Q)\rangle$: clearly $m_1(H)=0$ since $H$ is cyclic, but $K(H)=K(P,Q)$ contains a root of unity of order $\ell^n$, which contradicts property $(\mu)_s$ for $A$ when $n$ is large enough. In particular, choosing for $B$ an abelian variety which satisfies property $(\mu)_s$ (for example an elliptic curve without CM, cf.~\cite{MR2643390}), this shows that $(\mu)_s$ needs not hold for a product when it holds for the single factors.
\end{example}



\section{Property $(\mu)_w$}
\subsection{Preliminaries}
We fix once and for all an embedding of $\overline{\mathbb{Q}}$ into $\mathbb{C}$, and consider the number field $K$ as a subfield of $\overline{\mathbb{Q}} \subseteq \mathbb{C}$. The letter $A$ denotes a fixed abelian variety over $K$; if $\ell$ is a prime number and $n$ is a positive integer, we write $G_{\ell^n}$ for the Galois group of $K(A[\ell^n])/K$ and $G_{\ell^\infty}$ for the Galois group of $K(A[\ell^\infty])/K$. Finally, we take the following definition for the Mumford-Tate group of $A$:
\begin{definition}
Let $K$ be a number field and $A/K$ be an abelian variety. Let $V$ be the $\mathbb{Q}$-vector space $H_1(A(\mathbb{C}),\mathbb{Q})$, equipped with its natural Hodge structure of weight $-1$. Also let $V_\mathbb{Z}=H_1(A(\mathbb{C}),\mathbb{Z})$, write $\mathbb{S}:=\operatorname{Res}_{\mathbb{C}/\mathbb{R}}\left(\mathbb{G}_{m,\mathbb{C}} \right)$ for Deligne's torus, and let $h:\mathbb{S} \rightarrow \operatorname{GL}_{V \otimes \mathbb{R}}$ be the morphism giving $V$ its Hodge structure. We define $\MT(A)$ to be the $\mathbb{Q}$-Zariski closure of the image of $h$ in $\operatorname{GL}_V$, and extend it to a scheme over $\mathbb{Z}$ by taking its $\mathbb{Z}$-closure in $\operatorname{GL}_{V_{\mathbb{Z}}}$.
\end{definition}
\begin{remark}
Taking the $\mathbb{Z}$-Zariski closure in the previous definition allows us to consider points of $\MT(A)$ with values in arbitrary rings. It is clear that the Mumford-Tate group of $A$, even in this integral version, is insensitive to field extensions of $K$: indeed, it is defined purely in terms of data that can be read off $A_\mathbb{C}$, namely its Hodge structure and its integral homology.
Notice that $\MT(A)_{\mathbb{Q}}$, being an algebraic group over a field of characteristic 0, is smooth by Cartier's theorem. It follows that $\MT(A)$ is smooth over an open subscheme of $\operatorname{Spec} \mathbb{Z}$.\end{remark}

The following theorem summarizes fundamental results, due variously to Serre \cite{SerreLettreRibet2}, Wintenberger \cite{MR1944805}, Deligne \cite[I, Proposition 6.2]{DeligneInclusion}, Borovo{\u\i} \cite{MR0352101} and Pjatecki{\u\i}-{\v{S}}apiro \cite{MR0294352}, on the structure of Galois representations arising from abelian varieties over number fields; see also \cite[§10]{2015arXiv150505620H} for a detailed proof of the last statement.
\begin{theorem}\label{thm_InclusionMT}
Let $K$ be a number field and $A/K$ be an abelian variety. 

There exists a finite extension $L$ of $K$ such that for all primes $\ell$ the image of the natural representation
$
\rho_{\ell^\infty} : \abGal{L} \to \operatorname{Aut} T_\ell(A)
$
lands into $\MT(A)(\mathbb{Z}_\ell)$, and likewise the image of $\rho_{\ell} : \abGal{L} \to \operatorname{Aut} A[\ell]$ lands into $\MT(A)(\mathbb{F}_\ell)$. If furthermore the Mumford-Tate conjecture holds for $A$, then the index $[\MT(A)(\mathbb{Z}_\ell):\operatorname{Im} \rho_{\ell^\infty}]$ is bounded by a constant independent of $\ell$; the same is true for $[\MT(A)(\mathbb{F}_\ell):\operatorname{Im} \rho_\ell]$.
\end{theorem}

\subsection{Known results towards the Mumford-Tate conjecture}
While theorem \ref{thm_InclusionMT} will prove useful in establishing theorem \ref{thm_ms}, for the proof of theorem \ref{thm_mw} we shall also need some results which are known to hold independently of the truth of the Mumford-Tate conjecture, and which we now recall. The crucial point is that, even though we do not know in general that the Zariski closure of $G_{\ell^\infty}$ is ``independent of $\ell$'' in the sense predicted by the Mumford-Tate conjecture, results of Serre and Wintenberger imply that $G_{\ell^\infty}$ is not very far from being the group of $\mathbb{Z}_\ell$-points of an algebraic group. This is made more precise in the following theorem, for which we need to set some notation. Let $A/K$ be an abelian variety over a number field, and for every prime $\ell$ let $\underline{H}_\ell$ be the identity component of the $\mathbb{Z}_\ell$-Zariski closure of $G_{\ell^\infty}$. The groups $\underline{H}_\ell$ turn out to be reductive, except for finitely many primes $\ell$; when $\underline{H}_\ell$ is indeed reductive, we write $\underline{S}_\ell$ for its derived subgroup and $\underline{C}_\ell$ for its center. Following \cite{MR1944805}, we shall denote the special fiber of $\underline{H}_\ell$ (resp.~$\underline{S}_\ell$, $\underline{C}_\ell$) by $\SpecialFiber{H}$ (resp.~$\SpecialFiber{S}$, $\SpecialFiber{C}$), and the general fiber by $H_\ell$ (resp.~$S_\ell, C_\ell$). We then have the following result:

\begin{theorem}{(Serre \cite{SerreLettreRibet1981, SerreLettreRibet2, LettreVigneras}, Wintenberger \cite{MR1944805})}\label{thm_UnconditionalMT}
The following hold:
\begin{enumerate}
\item all the $\underline{H}_\ell$ but a finite number are smooth, reductive groups over $\mathbb{Z}_\ell$;
\item there is a finite extension $K'$ of $K$ with the property that for every prime $\ell$ the group 
$
\rho_{\ell^\infty}\left(\abGal{K'}\right)$ is contained in $\underline{H}_\ell(\mathbb{Z}_\ell)$;
\item the index
$
\left[\underline{H}_\ell(\mathbb{Z}_\ell) : \rho_{\ell^\infty}\left(\abGal{K'}\right)\right]
$
is bounded by a constant independent of $\ell$;
\item for all primes $\ell$ but finitely many exceptions, the special fiber $\SpecialFiber{H}$ of $\underline{H}_\ell$ acts semi-simply on $A[\ell]$, and the same is true for the special fiber $\SpecialFiber{S}$ of $\underline{S}_\ell$;
\item there exist an integer $N$ and a $\mathbb{Z}[1/N]$-subtorus $\underline{C}$ of $\operatorname{GL}_{2g,\mathbb{Z}[1/N]}$, containing the torus of homotheties, with the following property: for all primes $\ell$ not dividing $N$, the center $\underline{C}_\ell$ of $\underline{H}_\ell$ can be identified (up to conjugation) with $\underline{C} \times_{\mathbb{Z}[1/N]} \mathbb{Z}_\ell$.
\end{enumerate}

\end{theorem}
\begin{proof}
Part (1) follows from \cite[Theorem 1]{MR1944805} upon applying results of Zarhin \cite{MR553707}, as explained in \cite[§2.1]{MR1944805}, while (2) is a theorem of Serre \cite{SerreLettreRibet1981}. Part (3) follows from the main result of \cite{MR1944805} (which describes the derived subgroup of $\underline{H}_\ell$) together with the arguments of \cite{SerreLettreRibet2} (a description of the center of $\underline{H}_\ell$), cf.~\cite[§10]{2015arXiv150505620H} for a detailed proof. Part (4) is a consequence of the fundamental results of Faltings \cite{MR718935}, as it is again explained in \cite[§2.1]{MR1944805} (cf.~also \cite[§3.a]{LettreVigneras}).
Finally, (5) follows from Serre's theory of abelian representations: a detailed proof can be found in \cite{Vasiu}, see also \cite{Ullmo} and \cite[§10]{2015arXiv150505620H}.
\end{proof}

The next result we recall, again due to Serre, further implies that, even though we cannot show that the groups $\underline{H}_\ell$ are ``all the same'' (that is, that they all come from $\operatorname{MT}(A)$ by extension of scalars), their special fibers cannot vary too wildly:
\begin{theorem}{(Serre \cite[§1]{LettreVigneras})}\label{thm_FinitelyManyModels}
There exist a constant $c(g)$, depending only on $g=\dim A$, and finitely many $\mathbb{Z}$-algebraic subgroups $\underline{J}_1$, \ldots, $\underline{J}_k$ of $\operatorname{GL}_{2g,\mathbb{Z}}$ (again depending only on $g$) with the following property: if $\ell$ is a prime larger than $c(g)$ and $\SpecialFiber{H}$ acts semisimply on $A[\ell]$, then the $\overline{\mathbb{F}_\ell}$-algebraic group $\underline{S}_\ell \times_{\mathbb{Z}_\ell} \overline{\mathbb{F}_\ell} $ is $\operatorname{GL}_{2g,\overline{\mathbb{F}_\ell}}$-conjugate to one of the finitely many groups $\underline{J}_{1}\times_{\mathbb{Z}} \overline{\mathbb{F}_\ell}, \ldots, \underline{J}_{k} \times_{\mathbb{Z}} \overline{\mathbb{F}_\ell}$.
\end{theorem}



In view of the previous two theorems we introduce the following definition:
\begin{definition}\label{def_BadPrimes}
Let $A$ be an abelian variety over a number field $K$ and let $N$ be as in part (5) of theorem \ref{thm_UnconditionalMT}.
We shall say that a prime $\ell$ is \textit{bad} (for $A/K$) if any of the following is true: $\underline{H}_\ell$ is not smooth reductive over $\mathbb{Z}_\ell$, $\SpecialFiber{H}$ or $\SpecialFiber{S}$ does not act semisimply on $A[\ell]$, $\ell$ divides $N$, $\ell \leq c(\dim A)$ (with $c$ as in theorem \ref{thm_FinitelyManyModels}), or $\ell$ is ramified in $K$. Theorem \ref{thm_UnconditionalMT} ensures that for a given abelian variety there are only finitely many bad primes, and we call all the other primes \textit{good}.

\end{definition}



%

\subsection{Proof of theorem \ref{thm_mw}: preliminary reductions}
As the statement of theorem \ref{thm_mw} is clearly invariant under extension of the base field, parts (2) and (3) of theorem \ref{thm_UnconditionalMT} allow us to assume that $\rho_{\ell^\infty}\left(\abGal{K} \right)$ is included in $\underline{H}_\ell(\mathbb{Z}_\ell)$ for all primes $\ell$, in such a way that the index
$
 [\underline{H}_\ell(\mathbb{Z}_\ell):\operatorname{Gal}\left(K(A[\ell^\infty])/K \right)]
$
is bounded by a constant independent of $\ell$. 
Since the statement of theorem \ref{thm_mw} is also invariant under isogenies, making a further extension of the base field if necessary we can also assume without loss of generality that $A$ is principally polarized, which implies that $G_{\ell^\infty}$, resp.~$G_\ell$, is a subgroup of $\operatorname{GSp}_{2g}(\mathbb{Z}_\ell)$, resp.~of $\operatorname{GSp}_{2g}(\mathbb{F}_\ell)$. The definition of $\underline{H}_\ell$ then shows that we have inclusions $\underline{H}_\ell \subseteq \operatorname{GSp}_{2g,\mathbb{Z}_\ell}$ and $\SpecialFiber{H} \subseteq \operatorname{GSp}_{2g,\mathbb{F}_\ell}$.

The following simple lemma shows that the property of having index bounded by a constant is stable under passage to subgroups and quotients: knowing this will be useful to convert statements concerning the algebraic groups $\underline{H}_\ell$ into statements involving Galois groups, and vice versa.
\begin{lemma}\label{lemma_Index}
Let $C$ be a group and $A,B$ be subgroups of $C$ such that $[C:B]$ is finite. We have $[A:B\cap A ] \leq [C:B]$. Moreover, if $\pi:C \to D$ is a quotient of $C$, then $[D:\pi(B)] \bigm\vert [C:B]$.
\end{lemma}
\begin{proof}
The map $A \hookrightarrow C \to C/B$ induces an injection (of sets) of $A/(A \cap B)$ into $C/B$. The second statement is obvious.
\end{proof}


This easy fact allows us to work with ``equalities up to a finite index", for which we now introduce some notations. If $L_1,L_2$ are number fields that depend on $A/K$ and on some other set of parameters, we write $L_1 \uguale L_2$ to mean that there exists a constant $C$ (depending on $A/K$ only) such that the inequalities
$
[L_1:L_1 \cap L_2 ] \leq C \text{ and } [L_2 :L_1  \cap L_2 ] \leq C
$
%
hold for all values of the parameters; likewise, if $G_1, G_2$ are subgroups of a same group (and depend on some set of parameters), we write $G_1 \uguale G_2$ if both $[G_1:G_1 \cap G_2]$ and $[G_2:G_1 \cap G_2]$ are bounded by a constant depending only on $A/K$, uniformly in all other parameters. Furthermore, for two functions $f, g : I \to \mathbb{R}^+$, where $I$ is any set, we write $f \uguale g$ if there is a constant $C'>0$ such that $\frac{1}{C'} g(x) \leq f(x) \leq C' g(x)$  for all $x \in I$. Finally, to deal with arithmetic functions we introduce the following definition:
\begin{definition}
Let $\mathcal{P}$ be the set of prime numbers, $I$ be any set and $h:I \times \mathcal{P} \to \mathbb{N}^+$ be any function. We say that $h(x,\ell)$ is a power of $\ell$ up to a bounded constant if there exists a $C''>0$ such that for all $x \in I$ and $\ell \in \mathcal{P}$ we have $\frac{h(x,\ell)}{\ell^{v_\ell(h(x,\ell))}} \leq C''$, or equivalently, if the prime-to-$\ell$ part of $h(x,\ell)$ is bounded independently of $x$ and $\ell$.
\end{definition}
As a typical example of the use of this notation, for an abelian variety that satisfies the Mumford-Tate conjecture the conclusion of theorem \ref{thm_InclusionMT} can be expressed by writing $\operatorname{Gal}\left(K(A[\ell^\infty])/K \right) \uguale \MT(A)(\mathbb{Z}_\ell)$ and $\operatorname{Gal}\left(K(A[\ell])/K \right) \uguale \MT(A)(\mathbb{F}_\ell)$, while theorem \ref{thm_UnconditionalMT} implies that, for \textit{any} abelian variety $A$ over a number field $K$, possibly after replacing $K$ with a finite extension $K'$ we have 
$
\operatorname{Gal}\left( K(A[\ell])/K \right) \uguale \underline{H}_\ell(\mathbb{F}_\ell).
$
We can also apply lemma \ref{lemma_Index} to the groups $C=\underline{H}_\ell(\mathbb{F}_\ell)$, $B=\operatorname{Gal}\left(K(A[\ell])/K \right)$ and $A=\left\{x \in \underline{H}_\ell(\mathbb{F}_\ell) \bigm\vert xh=h \quad \forall h \in H \right\}$ to get
\[
\operatorname{Gal}\left(K(A[\ell])/ K(H) \right) \uguale \left\{ x \in \underline{H}_\ell(\mathbb{F}_\ell) \bigm\vert xh=h \quad \forall h \in H \right\},
\]
where the implied constant depends on $A/K$, but not on $\ell$ or $H$. Finally, notice that if $A,B$ are groups (depending on some set of parameters) such that $[B:A] \leq N$ for all values of the parameters, then taking $N':=N!$ we have $[B:A] \bigm\vert N'$, again for any choice of the parameters: if we so desire we can therefore replace boundedness conditions by divisibility conditions.

\subsection{Smoothness}\label{sect_Stabilizers}
In the course of the proof of theorem \ref{thm_mw} we shall need to know that certain algebraic groups are smooth; in this section we collect the relevant results in this direction.
Let $H$ be a finite subgroup of $A[\ell^\infty]$.
Write $H$ as $\prod_{i=1}^{2g} \mathbb{Z}/\ell^{m_i}\mathbb{Z}$ for certain integers $m_1 \geq \ldots \geq m_{2g}$, let $e_1, \ldots, e_{2g}$ be generators of the cyclic factors of $H$ (so $e_i$ is a torsion point of order $\ell^{m_i}$), and let $\widehat{e}_1, \ldots, \widehat{e}_{2g}$ be a basis of $T_\ell A$ lifting the $e_i$ (that is, satisfying $\widehat{e}_i \equiv e_i \pmod{\ell^{m_i}}$ for $i=1,\ldots,2g$). For a subset $I$ of $\left\{1,\ldots,2g\right\}$ we let $\underline{\mathcal{G}}_I$ be the $\mathbb{Z}_\ell$-algebraic group given by
\[
\underline{\mathcal{G}}_I = \left\{ M \in \underline{H}_\ell \bigm\vert M \widehat{e}_i = \widehat{e}_i \quad \forall i \in I \right\}.
\]
We plan to show that $\underline{\mathcal{G}}_I$ and various other related groups are smooth (over $\mathbb{Z}_\ell$, or equivalently over $\mathbb{F}_\ell$, cf.~lemma \ref{lemma_SmoothOverFlImpliesSmoothOverZl}) whenever $\ell$ is sufficiently large with respect to $A/K$, independently of the choice of $\widehat{e_1}, \ldots, \widehat{e_{2g}}$ and $I$ (the result crucial to our applications is lemma \ref{lemma_SmoothnessNew}). We shall make repeated use of the following fact:
\begin{theorem}\label{thm_ConditionsSmoothness}
Let $\ell$ be a prime number and $k$ be a finite field of characteristic $\ell$. Let $\mathcal{F}$ be an affine group scheme over $k$ with coordinate ring $R$. The following are equivalent:
\begin{enumerate}
\item $\mathcal{F}$ is smooth;
\item $R \otimes_{k} \overline{k}$ is reduced;
\item the nilpotency index of $R \otimes_{k} \overline{k}$ is smaller than $\ell$, that is, there exists an integer $e < \ell$ such that for all $a \in R\otimes_{k} \overline{k}$ and all positive integers $n$, the equality $a^n=0$ implies $a^e=0$;
\item the equality $\dim_{k} \operatorname{Lie} \mathcal{F} = \dim \mathcal{F}$ holds.
\end{enumerate}
\end{theorem}
\begin{proof}
1 and 2 are equivalent by \cite[Theorem on p.~88]{MR547117}. 1 and 4 are equivalent by \cite[Corollary on p.~94]{MR547117}. Clearly 2 implies 3, and 3 implies 2 by the same argument that proves Cartier's theorem (all algebraic groups over a field of characteristic zero are smooth), see for example \cite[Proof of Theorem 10.1]{milneAGS}.
\end{proof}


The following proposition, while certainly well-known to experts, does not seem to appear anywhere in the literature; we will use it as a substitute for Cartier's theorem on smoothness when working over a field of positive characteristic.

\begin{proposition}\label{prop_PositiveCharCartier}
Let $n,d,m$ be fixed positive integers. There is a constant $c(n,d,m)$ with the following property: for every prime $\ell > c(n,d,m)$, every finite field $k$ of characteristic $\ell$, and every algebraic subgroup $\mathcal{F}$ of $\operatorname{GL}_{n,k}$ that is cut in $\displaystyle \frac{k[x_{ij},y]}{(\det(x_{ij})y-1)}$ by at most $m$ equations of degree at most $d$ is smooth over $k$.
\end{proposition}
\begin{proof}
Let $I=(f_1,\ldots,f_t)$ be the ideal defining $\mathcal{F}$ in $\displaystyle \frac{k[x_{ij},y]}{(\det(x_{ij})y-1)}$, where $t \leq m$ and the total degree of every $f_h$ is at most $d$. Let $\displaystyle R=\frac{k[x_{ij},y]}{(\det(x_{ij})y-1,I)}$ be the coordinate ring of $\mathcal{F}$.
To test smoothness we can base-change to $\overline{k}$, and by theorem \ref{thm_ConditionsSmoothness} we only need to prove that the nilpotency index of 
$
R \otimes_{k} \overline{k} \cong \displaystyle \frac{\overline{k}[x_{ij},y]}{(\det(x_{ij})y-1,f_1,\ldots,f_t)}$
is bounded by a function of $n$, $d$ and $m$ alone, uniformly in $\ell$ and $k$.
Now just notice that the ideal $\left(\det(x_{ij})y-1,f_1,\ldots,f_t\right)$ is generated by equations whose number and degree are bounded in terms of $n$, $d$, and $m$, so the result follows from \cite[Theorem 1.3]{MR2198324} (see also \cite{MR944576}). 
More precisely, since we have at most $m+1$ equations of degree at most $\max\{d,n+1\}$, \cite[Theorem 1.3]{MR2198324} shows that one can take $c(n,d,m)=\max\{d,n+1\}^{m+1}$.
\end{proof}

\begin{lemma}
Let $n$ be a positive integer, $\mathcal{F}$ be a group subscheme of $\operatorname{GL}_{n,\mathbb{Q}_\ell}$, and let $\underline{\mathcal{F}}$ be the Zariski closure of $\mathcal{F}$ in $\operatorname{GL}_{n,\mathbb{Z}_\ell}$. Then $\underline{\mathcal{F}}$ is flat over $\operatorname{Spec} \mathbb{Z}_\ell$.
\end{lemma}
\begin{proof}
An affine scheme $\operatorname{Spec} \underline{R}$ over $\mathbb{Z}_\ell$ is flat if and only if its coordinate ring $\underline{R}$ is a torsion-free $\mathbb{Z}_\ell$-module (\cite[Corollary 2.14]{MR1917232}). In our case, if $I$ is the ideal of $\frac{\mathbb{Q}_\ell[x_{ij},y]}{(\det(x_{ij})y-1)}$ that defines $\mathcal{F}$, then $\underline{I}:=I \cap \frac{\mathbb{Z}_\ell[x_{ij},y]}{(\det(x_{ij})y-1)}$  is the ideal defining $\underline{\mathcal{F}}$. In particular, the coordinate ring $\underline{R}$ of $\underline{\mathcal{F}}$ injects into the coordinate ring $R$ of $\mathcal{F}$, which is torsion-free since it is a $\mathbb{Q}_\ell$-vector space.
\end{proof}

\begin{lemma}\label{lemma_SmoothOverFlImpliesSmoothOverZl}
Let $n$ be a positive integer, $\mathcal{F}$ be a group subscheme of $\operatorname{GL}_{n,\mathbb{Q}_\ell}$, and let $\underline{\mathcal{F}}$ be the Zariski closure of $\mathcal{F}$ in $\operatorname{GL}_{n,\mathbb{Z}_\ell}$. Suppose furthermore that $\underline{\mathcal{F}}$ is smooth over $\mathbb{F}_\ell$: then $\underline{\mathcal{F}}$ is smooth over $\mathbb{Z}_\ell$.
\end{lemma}
\begin{proof}
In order for a scheme $\underline{\mathcal{F}}\bigm/\operatorname{Spec} \mathbb{Z}_\ell$ to be smooth, it is necessary and sufficient that it is locally finitely presented and flat, with fibers that are smooth varieties all of the same dimension. Finite presentation is obvious in our context, and flatness follows from the previous lemma. The dimension of the fibers is locally constant by flatness, hence constant since the only open subset of $\operatorname{Spec} \mathbb{Z}_\ell$ containing the closed point is all of $\operatorname{Spec} \mathbb{Z}_\ell$. It remains to show smoothness of the fibers: the generic fiber is smooth by Cartier's theorem (\cite[§11.4]{MR547117}), and the special fiber is smooth by assumption.
\end{proof}

We finally come to the central result of this section:

\begin{lemma}\label{lemma_SmoothnessNew}
For all $\ell$ sufficiently large (depending only on $A/K$), for all $\mathbb{Z}_\ell$-bases $\widehat{e}_1, \ldots, \widehat{e}_{2g}$ of $T_\ell A$, and for all subsets $I$ of $\{1,\ldots,2g\}$, the stabilizer $\underline{\mathcal{G}}_I$ in $\underline{H}_\ell$ of the vectors $\widehat{e_i}$ (for $i \in I$) is smooth over $\mathbb{Z}_\ell$.
\end{lemma}
\begin{proof}
Notice first that $\underline{\mathcal{G}}_I$ can be obtained as the $\mathbb{Z}_\ell$-Zariski closure of the $\mathbb{Q}_\ell$-group scheme
\[
\left\{ M \in H_\ell \bigm\vert M\widehat{e_i}=\widehat{e_i} \; \forall i \in I \right\}.
\]
By lemma \ref{lemma_SmoothOverFlImpliesSmoothOverZl} it then suffices to prove smoothness over $\mathbb{F}_\ell$, and to do this we can base-change to $\overline{\mathbb{F}_\ell}$. We can also assume that $\ell$ is a good prime (cf.~definition \ref{def_BadPrimes}).
 By theorems \ref{thm_UnconditionalMT} and \ref{thm_FinitelyManyModels} there are algebraic subgroups $\mathcal{S}:=(\underline{J}_i)_{\overline{\mathbb{F}_\ell}}$ and $\mathcal{C}:=\underline{C} _{\overline{\mathbb{F}_\ell}}$ of $\operatorname{GL}_{2g,\overline{\mathbb{F}_\ell}}$ such that $(\underline{H}_\ell)_{\overline{\mathbb{F}_\ell}}$ is reductive, with center conjugated to $\mathcal{C}$ and derived subgroup conjugated to $\mathcal{S}$. In particular, we can find isomorphisms $\varphi_C : \mathcal{C} \to (\underline{C}_\ell)_{\overline{\mathbb{F}_\ell}}$ and $\varphi_S : \mathcal{S} \to (\underline{S}_\ell)_{\overline{\mathbb{F}_\ell}}$ that are given by conjugation by an element of $\operatorname{GL}_{2g}(\overline{\mathbb{F}_\ell})$, and consider the map
\[
\begin{array}{cccc}
p : & \mathcal{C} \times \mathcal{S} &  \to & \left(\underline{H}_\ell\right)_{\overline{\mathbb{F}_\ell}} \\
& (c,s) & \mapsto & \varphi_C(c)\varphi_S(s).
\end{array}
\]
Notice that $p$ is given by the composition of the morphism $(\varphi_C,\varphi_S)$ with the multiplication map $m:\operatorname{GL}_{2g,\overline{\mathbb{F}_\ell}} \times \operatorname{GL}_{2g,\overline{\mathbb{F}_\ell}} \to \operatorname{GL}_{2g,\overline{\mathbb{F}_\ell}}$. Observe further that the polynomials defining $m$ are clearly independent of $\ell$, because $m$ comes from base-change from the universal multiplication map $m:\operatorname{GL}_{2g,\mathbb{Z}} \times \operatorname{GL}_{2g,\mathbb{Z}} \to \operatorname{GL}_{2g,\mathbb{Z}}$. Moreover, since $\varphi_C$ and $\varphi_S$ are simply given by linear changes of basis, also the polynomials defining $\varphi_C$ and $\varphi_S$ have degree bounded independently of $\ell$. It follows that the polynomials defining $p$ have degree bounded independently of $\ell$.

Consider now the pullback $\mathcal{F}:=p^*\left((\underline{\mathcal{G}}_I)_{\overline{\mathbb{F}_\ell}}\right) \subseteq \mathcal{C} \times \mathcal{S}$: since $(\underline{\mathcal{G}}_I)_{\overline{\mathbb{F}_\ell}} \hookrightarrow (\underline{H}_\ell)_{\overline{\mathbb{F}_\ell}}$ is a closed embedding, $\mathcal{F} \hookrightarrow \mathcal{C} \times \mathcal{S}$ is again a closed embedding. We claim that $\mathcal{F}$, as a subgroup of $\operatorname{GL}_{2g,\overline{\mathbb{F}_\ell}} \times \operatorname{GL}_{2g,\overline{\mathbb{F}_\ell}} \subseteq \operatorname{GL}_{4g,\overline{\mathbb{F}_\ell}}$, is defined by equations whose number and degree are bounded independently of $\ell$ and of the vectors $\widehat{e_i}$. To see this, notice first that $\mathcal{C} \times \mathcal{S}$ is defined by equations bounded in number and degree -- indeed, up to a linear change of coordinates (which does not alter neither the number nor the total degree of the involved polynomials), these are the same equations that define $\underline{C}$ and the group $\underline{J}_i$ over $\mathbb{Z}$, and there are only finitely many groups $\underline{J}_i$ to consider. Next remark that the conditions $M \widehat{e_i} = \widehat{e_i}$ that define $\underline{\mathcal{G}}_I$ in $\underline{H}_\ell$ are given in coordinates by no more than $(2g)^2$ linear equations ($2g$ linear equations for each vector, and at most $2g$ vectors), each of which pulls back via $p^*$ to a single equation in the coordinate ring of $\operatorname{GL}_{4g,\overline{\mathbb{F}_\ell}}$. Finally, the degree of these equations is bounded independently of $\ell$, since it only depends on the degrees of the polynomials defining $p$, which as already proved are independent of $\ell$. It follows from proposition \ref{prop_PositiveCharCartier} that for $\ell$ large enough $\mathcal{F}$ is smooth, hence its coordinate ring is reduced. Finally, notice that $p$ induces an injection of the coordinate ring of $(\underline{\mathcal{G}}_I)_{\overline{\mathbb{F}_\ell}}$ in that of $\mathcal{F}$, so since the latter is reduced the same is true for the former: $(\underline{\mathcal{G}}_I)_{\overline{\mathbb{F}_\ell}}$ is then smooth by theorem \ref{thm_ConditionsSmoothness}.
\end{proof}

An easy variant of the previous proof also yields:

\begin{lemma}\label{lemma_SmoothnessNew2}
Let $\lambda : \operatorname{GSp}_{2n,\mathbb{Z}_\ell} \to \mathbb{G}_{m,\mathbb{Z}_\ell}$ be the (algebraic) multiplier character. With the notation of the previous lemma, the $\mathbb{Z}_\ell$-algebraic group
\[
\underline{\mathcal{G}}_I^{(1)}=\left\{ M \in \underline{H}_\ell \bigm\vert M h = h \quad \forall h \in H, \; \lambda(M)=1 \right\}
\]
is smooth over $\mathbb{Z}_\ell$ for all $\ell$ larger than some bound that only depends on $A/K$.
\end{lemma}
\begin{proof}
Arguing as in the proof of lemma \ref{lemma_SmoothnessNew}, it suffices to show that $p^* \left(\underline{\mathcal{G}}_I^{(1)}\right)_{\overline{\mathbb{F}_\ell}}$ is defined by equations whose number and degree are bounded independently of $\ell$, of $\widehat{e_i}$, and of $I$. 
This follows easily from the same argument as in the previous proof, because the equations defining $p^* \left(\underline{\mathcal{G}}_I^{(1)}\right)_{\overline{\mathbb{F}_\ell}}$ are the same as those defining $p^* \left(\underline{\mathcal{G}}_I\right)_{\overline{\mathbb{F}_\ell}}$, together with the single equation $\lambda(M)-1=0$, which is given by a polynomial whose degree is independent of $\ell$: indeed, the morphism $\lambda$ comes by base-change from a certain universal morphism $\lambda:\operatorname{GSp}_{2g,\mathbb{Z}} \to \mathbb{G}_{m,\mathbb{Z}}$, hence the polynomial that defines it does not depend on $\ell$.
\end{proof}

\begin{definition}\label{def_VeryGood} We shall say that the prime $\ell$ is \textit{very good} for $A/K$ if it is good and so large that all the groups $\underline{\mathcal{G}}_I$ and $\underline{\mathcal{G}}_I^{(1)}$ are smooth over $\mathbb{Z}_\ell$, for every $\mathbb{Z}_\ell$-basis of $T_\ell A$ and every subset $I$ of $\{1,\ldots,2g\}$.
\end{definition}


\subsection{Connected components}
In this section we show that the groups we are interested in have a bounded number of connected components, and relate this number to certain cohomology groups.

Recall from the previous section the notation $\underline{\mathcal{G}}_I$:  given a $\mathbb{Z}_\ell$-basis $\widehat{e}_1, \ldots, \widehat{e}_{2g}$ of $T_\ell A$ and a subset $I$ of $\{1,\ldots,2g\}$, the $\mathbb{Z}_\ell$-algebraic group $\underline{\mathcal{G}}_I$ is the stabilizer in $\underline{H}_\ell$ of the vectors $\widehat{e_i}$ for $i \in I$.

\begin{lemma}\label{lemma_ConnectedComponents}
There is a constant $B$, depending only on $A/K$, with the following property. 
For all primes $\ell$ that are good for $A$ and for all subgroups $H$ of $A[\ell]$, the number of connected components of
\[
\StabilizerAlg = \left\{ M \in \SpecialFiber{H} \bigm\vert Mh=h \quad \forall h \in H \right\}=(\underline{\mathcal{G}}_I)_{\mathbb{F}_\ell}
\]
does not exceed $B$.
\end{lemma}
\begin{proof}
Notice first that it is enough to bound the number of $\overline{\mathbb{F}_\ell}$-points of the group of components of $\StabilizerAlg$, hence it is enough to consider the number of irreducible components of $\StabilizerAlg_{\overline{\mathbb{F}_\ell}}$. As in the proof of lemma \ref{lemma_SmoothnessNew}, we consider the pullback $p^*\StabilizerAlg_{\overline{\mathbb{F}_\ell}} \subseteq \mathcal{C} \times \mathcal{S} \subseteq \operatorname{GL}_{4g,\overline{\mathbb{F}_\ell}}$, and remark that since $p^*\StabilizerAlg_{\overline{\mathbb{F}_\ell}} \to \StabilizerAlg_{\overline{\mathbb{F}_\ell}}$ is onto, it suffices to bound the number of irreducible components of $p^*\StabilizerAlg_{\overline{\mathbb{F}_\ell}}$. Again as in the proof of lemma \ref{lemma_SmoothnessNew}, we know that $p^*\StabilizerAlg_{\overline{\mathbb{F}_\ell}}$ is defined by equations whose number and degree are bounded independently of $\ell$ and $H$.

By a variant of Bézout's theorem (see \cite[Theorem 7.1]{MR1390322} for a precise statement), this implies that the number of irreducible components of $p^*\StabilizerAlg_{\overline{\mathbb{F}_\ell}}$ is bounded uniformly in $\ell$ and $H$, hence the same is true for the number of connected components of $\StabilizerAlg_{\overline{\mathbb{F}_\ell}}$, whence a constant $B$ such that $|\StabilizerAlg/\Smooth| \leq B$ for all good primes $\ell$ and all subgroups $H$ of $A[\ell]$.
\end{proof}

Similarly to what we did with lemmas \ref{lemma_SmoothnessNew} and \ref{lemma_SmoothnessNew2}, a simple variant of the same argument shows
\begin{lemma}\label{lemma_ConnectedComponents2}
There is a constant $B_1$, depending only on $A/K$, with the following property. 
For all primes $\ell$ that are good for $A$ and for all subgroups $H$ of $A[\ell]$, the number of connected components of
\[
\StabilizerAlg_1 = \left\{ M \in \SpecialFiber{H} \bigm\vert Mh=h \quad \forall h \in H, \; \lambda(M)=1 \right\}=\left(\underline{\mathcal{G}}_I^{(1)}\right)_{\mathbb{F}_\ell}
\]
does not exceed $B_1$.
\end{lemma}

\begin{lemma}
Let $\mathcal{G}$ be a finite étale group scheme of order $N$ over $\mathbb{F}_\ell$. The first cohomology group $H^1(\mathbb{F}_\ell,\mathcal{G})$ is finite, of order not exceeding $N$.
\end{lemma}
\begin{proof}
Recall (\cite[§6.4]{MR547117}) that the association $\mathcal{G} \mapsto \mathcal{G}(\overline{\mathbb{F}_\ell})$ establishes an equivalence between the category of étale group schemes over $\mathbb{F}_\ell$ and that of finite groups with a continuous action of $\operatorname{Gal}\left(\overline{\mathbb{F}_\ell}/\mathbb{F}_\ell \right)$. To prove the lemma it is thus enough to consider the cohomology $H^1(\mathbb{F}_\ell,G)$ of a finite group $G$ of order $N$ equipped with a continuous action of $\hat{\mathbb{Z}} \cong \operatorname{Gal}\left(\overline{\mathbb{F}_\ell}/\mathbb{F}_\ell \right)$. An element of $H^1\left(\hat{\mathbb{Z}} , G\right)$ is represented by a continuous map $\hat{\mathbb{Z}}\to G$, which in turn is uniquely determined by the image of a topological generator of $\hat{\mathbb{Z}}$: it follows that there are no more than $|G|=N$ such maps, hence that the order of $H^1(\hat{\mathbb{Z}},G)$ is bounded by $N$ as claimed.
\end{proof}

\begin{lemma}\label{lemma_Cohomology}
Let $\mathcal{G}$ be a linear algebraic group over $\mathbb{F}_\ell$. The order of $H^1(\mathbb{F}_\ell,\mathcal{G})$ is at most the order of $H^1(\mathbb{F}_\ell,\mathcal{G}/\mathcal{G}^0)$, so in particular the order of $H^1(\mathbb{F}_\ell,\mathcal{G})$ does not exceed the order of the group of components of $\mathcal{G}$.
\end{lemma}
\begin{proof}
The long exact sequence in cohomology associated with the sequence
\[
1 \to \mathcal{G}^0 \to \mathcal{G}\to \mathcal{G}/\mathcal{G}^0 \to 1
\]
contains the segment
$
H^1(\mathbb{F}_\ell,\mathcal{G}^0) \to H^1(\mathbb{F}_\ell,\mathcal{G}) \to H^1(\mathbb{F}_\ell,\mathcal{G}/\mathcal{G}^0),
$
where the first term is trivial by Lang's theorem (any connected algebraic group over a finite field has trivial $H^1$, \cite[Theorem 2]{MR0086367}). The first statement follows. The second is then a consequence of the previous lemma and of the fact that $\mathcal{G}/\mathcal{G}^0$ is étale by \cite[§6.7]{MR547117}.
\end{proof}

\subsection{Proof of theorem \ref{thm_mw}}
We now come to the core of the proof of theorem \ref{thm_mw}. Let $H$ be a finite subgroup of $A[\ell^\infty]$ of exponent $\ell^n$. As shown in \cite[Proposition 3.9]{MR2862374}, the degree $\left[ K(H) \cap K(\mu_{\ell^\infty}) : K \right]$ is closely related to the multipliers of automorphisms in $\operatorname{Gal}\left(K(A[\ell^n])/K(H) \right)$, thought of as elements of $\operatorname{GSp}_{2g}(\mathbb{Z}/\ell^n\mathbb{Z})$: through the next few lemmas we shall therefore investigate the image of the multiplier map when restricted to $\operatorname{Gal}\left(K(A[\ell^n])/K(H) \right)$.

\begin{lemma}\label{lemma_Core}
Let $A/K$ be an abelian variety over a number field. For all primes $\ell$ and for all finite subgroups $H$ of $A[\ell]$ there exists $m \in \left\{0,1\right\}$ such that
\[
\left[ K(\mu_{\ell^m}) : K \right] \uguale \left[ K(H) \cap K(\mu_{\ell}) : K \right],
\]
that is to say, there exists $D>0$ (depending on $A/K$) with the following property: for every $\ell$ and every subgroup $H$ of $A[\ell]$ there exists $m \in \left\{0,1\right\}$ such that
\begin{equation}\label{eq_WeakMu5}
D^{-1} \left[ K(H) \cap K(\mu_{\ell}) : K \right] \leq \left[ K(\mu_{\ell^m}) : K \right] \leq D \left[ K(H) \cap K(\mu_{\ell}) : K \right].
\end{equation}
\end{lemma}

\begin{proof}
Observe first that it suffices to prove that the conclusion of the lemma holds for all but finitely many primes: indeed, for a fixed prime $\ell$ the finite group $A[\ell]$ possesses only finitely many subgroups $H$, so we can choose $D$ so large that \eqref{eq_WeakMu5} holds for any such $H$ (with $m=0$, say). We can therefore assume that $\ell$ is very good (cf.~definition \ref{def_VeryGood}).
 Recall that $\SpecialFiber{H}$ is a subgroup of $\operatorname{GSp}_{2g,\mathbb{F}_\ell}$, so that there is a well-defined  multiplier character $\lambda: \SpecialFiber{H} \to \mathbb{G}_{m,\mathbb{F}_\ell}$.
At the level of $\mathbb{F}_\ell$-points we have $G_\ell \subseteq \underline{H}_\ell(\mathbb{F}_\ell) \subseteq \operatorname{GSp}_{2g}(\mathbb{F}_\ell)$, and -- since we assume $A$ to be principally polarized -- for all primes $\ell$ we have 
$ \lambda \circ \rho_\ell = \chi_\ell,
$
the mod-$\ell$ cyclotomic character.
Let now $e_1,\ldots,e_{2g}$ be an $\mathbb{F}_\ell$-basis of $A[\ell]$ such that $e_1,\ldots,e_r$ is an $\mathbb{F}_\ell$-basis of $H$. We consider the finite group $\Stabilizer=\left\{ M \in G_\ell \bigm\vert M \cdot h = h \quad \forall h \in H \right\}$, that is, the stabilizer of $H$ in $G_\ell$, and the algebraic group
$
\StabilizerAlg = \left\{ M \in \SpecialFiber{H} \bigm\vert M \cdot e_i = e_i, \; 1 \leq i \leq r \right\},
$
that is, the stabilizer of $H$ in $\SpecialFiber{H}$.
It is clear by definition that $\Stabilizer=G_\ell \cap \StabilizerAlg(\mathbb{F}_\ell)$; 
since $G_\ell \uguale \underline{H}_\ell(\mathbb{F}_\ell)$, this shows in particular that $T \uguale \StabilizerAlg(\mathbb{F}_\ell)$.
Notice that $\StabilizerAlg$ is smooth over $\mathbb{F}_\ell$: indeed, the group $\StabilizerAlg$ is the base-change to $\mathbb{F}_\ell$ of a corresponding group $\underline{\mathcal{G}}_I$ over $\mathbb{Z}_\ell$ (notation as in section \ref{sect_Stabilizers}), and is therefore smooth over $\mathbb{F}_\ell$ by virtue of lemma \ref{lemma_SmoothnessNew} and the fact that $\ell$ is very good. Furthermore, by lemma \ref{lemma_ConnectedComponents}, the group of components of $\StabilizerAlg$ has order bounded by a constant $B$ independent of $\ell$ and $H$. By lemma \ref{lemma_ConnectedComponents2}, the order of the group of connected components of the algebraic group
$
\StabilizerAlg_1 =\left\{ M \in H_\ell(\ell) \bigm\vert M \cdot h=h \quad \forall h \in H, \; \lambda(M)=1 \right\}=\ker (\lambda: \StabilizerAlg \to \mathbb{G}_{m,\mathbb{F}_\ell})
$
is also bounded by a constant independent of $\ell$ and $H$, which we call $B_1$, and furthermore $\StabilizerAlg_1$ is smooth since $\ell$ is very good. 
Finally, the group $\Stabilizer_1=\left\{ M \in G_\ell \bigm\vert M \cdot h=h \quad \forall h\in H, \; \lambda(M)=1 \right\}$ satisfies $\Stabilizer_1 \uguale \StabilizerAlg_1(\mathbb{F}_\ell)$.
Consider now the restriction of $\lambda : \operatorname{GSp}_{2g,\mathbb{F}_\ell} \to \mathbb{G}_{m,\mathbb{F}_\ell}$ to $\Smooth$, the identity component of $\StabilizerAlg$. As $\Smooth$ is smooth, the image $\lambda(\Smooth)$ is a connected reduced subgroup of $\mathbb{G}_{m,\mathbb{F}_\ell}$, hence it is either trivial or all of $\mathbb{G}_{m,\mathbb{F}_\ell}$. Let us consider the two cases separately.


\medskip

\noindent\textbf{$\lambda(\Smooth)$ is trivial.} As we have already remarked we have $\Stabilizer \subseteq \StabilizerAlg(\mathbb{F}_\ell)$. It follows that the order of $\lambda(\Stabilizer)$ is at most the order of $\lambda(\StabilizerAlg(\mathbb{F}_\ell))$, which in turn does not exceed $[\StabilizerAlg:\Smooth]$ since the restriction of $\lambda$ to $\Smooth$ is trivial. Hence we have $|\lambda(\Stabilizer)| \leq [\StabilizerAlg:\Smooth] \leq B$.

\smallskip

\noindent\textbf{$\lambda: \Smooth \to \mathbb{G}_{m,\mathbb{F}_\ell}$ is onto.} Consider the exact sequence
\[
1 \to \StabilizerAlg_1 \to \StabilizerAlg \xrightarrow{\lambda} \mathbb{G}_{m,\mathbb{F}_\ell} \to 1
\]
and take $\mathbb{F}_\ell$-rational points: the associated long exact sequence in cohomology shows that
$
\StabilizerAlg(\mathbb{F}_\ell) \xrightarrow{\lambda} \mathbb{G}_{m,\mathbb{F}_\ell}(\mathbb{F}_\ell)=\mathbb{F}_\ell^\times \to H^1\left(\mathbb{F}_\ell,\StabilizerAlg_1 \right)
$ is exact, so $\left|\operatorname{coker} \left(\StabilizerAlg(\mathbb{F}_\ell) \xrightarrow{\lambda} \mathbb{F}_\ell^\times\right)\right|$ is at most $\left|H^1\left(\mathbb{F}_\ell,\StabilizerAlg_1 \right) \right|$, which in turn (by lemmas \ref{lemma_Cohomology} and \ref{lemma_ConnectedComponents2}) does not exceed $B_1$.
Since $\Stabilizer \uguale \StabilizerAlg(\mathbb{F}_\ell)$, it follows that
$
|\lambda(\Stabilizer)| \uguale |\lambda(\StabilizerAlg(\mathbb{F}_\ell))| \geq \frac{\ell-1}{B_1},
$
that is, there exists a constant $B'$ (independent of $\ell$, as long as it is very good) such that whenever $\lambda:\Smooth \to \mathbb{G}_{m,\mathbb{F}_\ell}$ is onto the inequality $\displaystyle |\lambda(\Stabilizer)| \geq \frac{\ell-1}{B'}$ holds.

Let now $B''$ be a constant large enough that inequality \eqref{eq_WeakMu5} in the statement of the lemma holds, with $D=B''$, for all the (finitely many) primes $\ell$ that are not very good, and for the (finitely many) subgroups $H$ of $A[\ell]$, for each of these primes. Finally set $D=\max\left\{B,B',B'' \right\}$. We now show that inequality \eqref{eq_WeakMu5} is satisfied for all primes $\ell$ and all subgroups $H$ of $A[\ell]$. It is clear by construction that this is true for the primes that are not very good, so we can suppose that $\ell$ is unramified in $K$ and that $\StabilizerAlg$ and $\StabilizerAlg_1$ are smooth over $\mathbb{F}_\ell$. 
Observe that the group $\Stabilizer$ we considered above is by definition the Galois group of $K(A[\ell])/K(H)$, whereas the Galois group of $K(A[\ell])$ over $K(\mu_\ell)$ is $N:=\ker \left( G_\ell \stackrel{\lambda}{\longrightarrow} \mathbb{F}_\ell^\times \right)$. It follows that the Galois group of $K(A[\ell])$ over $K(H) \cap K(\mu_\ell)$ is the group generated by $\Stabilizer$ and $N$, hence the degree of $K(H) \cap K(\mu_\ell)$ over $K$ is the index of $N\Stabilizer$ in $G_\ell$. On the other hand we have $|G_\ell/N\Stabilizer| = \frac{\left|G_\ell/N\right|}{\left|N\Stabilizer/N\right|}$ (recall that $N$ is normal in $G_\ell$ by construction), and $G_\ell/N$ is isomorphic to the image of $\lambda : G_\ell \to \mathbb{F}_\ell^\times$. As $\ell$ is unramified in $K$, the mod-$\ell$ cyclotomic character $\chi_\ell : \abGal{K} \to \mathbb{F}_\ell^\times$ is surjective, hence we have $\lambda(G_\ell)=\chi_\ell(\abGal{K})=\mathbb{F}_\ell^\times$ and therefore
\[
[K(H) \cap K(\mu_\ell):K] = |G_\ell/N\Stabilizer| = \displaystyle \frac{|\lambda(G_\ell)|}{|\lambda(N\Stabilizer)|} = \frac{\ell-1}{|\lambda(\Stabilizer)|}.
\]
By our previous arguments we now see that
\begin{itemize}
\item either $\lambda(\Smooth)$ is trivial, in which case $1 \leq |\lambda(\Stabilizer)| \leq B$ and \eqref{eq_WeakMu5} is satisfied by taking $m=1$;
\item or $\lambda : \Smooth \to \mathbb{G}_{m,\mathbb{F}_\ell}$ is onto, in which case we have $\displaystyle\frac{\ell-1}{B'} \leq |\lambda(\Stabilizer)| \leq \ell-1$ and \eqref{eq_WeakMu5} is satisfied by taking $m=0$.
\end{itemize}
\end{proof}

\begin{remark} It is clear from the definitions that (if $\ell$ is large enough) the integer $m$ of the previous lemma satisfies $m \geq m_1(H[\ell])$. For the group $\mathcal{H}$ considered below in the proof of theorem \ref{thm_ms} we have $m_1(\mathcal{H})=0$ and $m=1$, which shows that equality needs not hold.
\end{remark}

To complete the proof of theorem \ref{thm_mw} we need two more lemmas.

\begin{lemma}\label{lemma_k0}
Let $K$ be a number field and $A/K$ be an abelian variety. For any finite subgroup $H$ of $A[\ell^\infty]$ the degree $[K(H):K(H[\ell])]$ is a power of $\ell$ (up to a bounded constant).
\end{lemma}
\begin{proof}
We use the notation from section \ref{sect_Stabilizers}; in particular we write $H \cong \prod_{i=1}^{2g} \mathbb{Z}/\ell^{m_i}\mathbb{Z}$, and fix generators $e_1,\ldots,e_{2g}$ of $H$ and a basis 
$\widehat{e}_1, \ldots, \widehat{e}_{2g}$ of $T_\ell A$ lifting the $e_i$.
We suppose first that $\ell$ is a very good prime. 
Inspired by the approach of \cite{MR2862374}, given $\mathbb{Z}_\ell$-algebraic subgroups $\mathcal{G}_1 \subseteq \mathcal{G}_2 \subseteq \cdots \subseteq \mathcal{G}_t$ of a $\mathbb{Z}_\ell$-group $\mathcal{G}$, a strictly increasing sequence $n_1 < n_2 < \cdots < n_t$ of positive integers, and a positive integer $n$, we now denote by $\mathcal{G}(n; n_1, \ldots, n_t)$ the finite group
\[
\left\{ M \in \mathcal{G}(\mathbb{Z}/\ell^n\mathbb{Z}) \bigm\vert M \in \mathcal{G}_i \text{ mod } \ell^{\min(n,n_i)}, \quad i=1,\ldots,t \right\}.
\]
It is natural to also consider case of $t$ being $0$: if $n_i$ is the empty sequence, we simply define $\mathcal{G}(n)=\mathcal{G}(\mathbb{Z}/\ell^n\mathbb{Z})$.
To $H$ we now attach a strictly decreasing sequence of positive integers $m^{(1)}>m^{(2)}>\cdots>m^{(t)} \geq 1$ (where $t \leq 2g$) by setting
\[
m^{(1)} = \max \left\{m_i \bigm\vert m_i \neq 0 \right\} \text{ and recursively } m^{(r+1)}=\max \left\{m_i  \bigm\vert 0 < m_i < m^{(r)} \right\},
\]
and, for $1 \leq r \leq t$, we let $I_r=\left\{ i \in \left\{1,\ldots,2g \right\} \bigm\vert m_i \geq m^{(r)} \right\}$. Finally, for $1 \leq r \leq t$, we set
\[
\mathcal{G}_r:=\underline{\mathcal{G}}_{I_{t+1-r}} = \left\{ M \in \underline{H}_\ell \bigm\vert M \cdot \widehat{e_i} = \widehat{e_i} \text{ for } i \in I_{t+1-r}\right\},
\]
and we consider the strictly \textit{increasing} sequence $n_r=m^{(t+1-r)}$ (for $1 \leq r \leq t$). 

By our assumptions on $\ell$ all the groups $\mathcal{G}_r$ are smooth over $\mathbb{Z}_\ell$, and, as in \cite{MR2862374}, we easily see that the $\mathcal{G}_r$ so defined form an increasing sequence of subgroups of $\mathcal{G}:=\underline{H}_\ell$ such that $\left[ K(H[\ell^m]): K \right] \uguale \left[\mathcal{G}(\mathbb{Z}/\ell^m\mathbb{Z}) :\mathcal{G}(m;n_1, \ldots, n_t)\right]$. We now show that (for any $H$ and any $m \geq 1$) the number
\begin{equation}\label{eq_IndexEll}
\frac{\left[\mathcal{G}(\mathbb{Z}/\ell^m\mathbb{Z}) :\mathcal{G}(m;n_1, \ldots, n_t)\right]}{\left[\mathcal{G}(\mathbb{Z}/\ell\mathbb{Z}) :\mathcal{G}(1;n_1, \ldots, n_t)\right]}
\end{equation}
is a power of $\ell$.
To prove this fact, we preliminarily show that for all $m \geq 2$ the reduction map $\mathcal{G}\left(\mathbb{Z}/\ell^m \mathbb{Z} \right) \xrightarrow{\pi_{m-1}} \mathcal{G}\left(\mathbb{Z}/\ell^{m-1} \mathbb{Z} \right)$ maps $\mathcal{G}(m;n_1,\ldots,n_t)$ surjectively onto $\mathcal{G}(m-1;n_1,\ldots,n_t)$. We can proceed by induction on $t$, showing the stronger statement that this is true for any chain of groups $\mathcal{G}_1\subset \mathcal{G}_2 \subset \cdots \subset \mathcal{G}_t \subset\mathcal{G}$ where each term is smooth over $\mathbb{Z}_\ell$. 
Indeed,
\begin{itemize}
\item for $t=0$ the claim follows from the smoothness of $\mathcal{G}$ and Hensel's lemma;
\item if $m \leq n_t$, then we have
$
\mathcal{G}(j;n_1,\ldots,n_t)=\mathcal{G}_t(j;n_1,\ldots,n_{t-1})
$
both for $j=m$ and $j=m-1$, so the claim follows from the induction hypothesis;
\item if $m>n_t$, then $\mathcal{G}\left(\mathbb{Z}/\ell^m \mathbb{Z} \right) \to \mathcal{G}\left(\mathbb{Z}/\ell^{m-1} \mathbb{Z} \right)$ is surjective by smoothness of $\mathcal{G}$, and furthermore, since by assumption we have $m-1 \geq n_t > n_{t-1} > \ldots > n_1$, any lift to $\mathcal{G}\left(\mathbb{Z}/\ell^m \mathbb{Z} \right)$ of a point in $\mathcal{G}(m-1;n_1,\ldots,n_t)$ belongs to $\mathcal{G}(m;n_1,\ldots,n_t)$, so that the induced map $\mathcal{G}(m;n_1,\ldots,n_t) \to \mathcal{G}(m-1;n_1,\ldots,n_t)$ is indeed surjective.
\end{itemize}
We now prove our claim that \eqref{eq_IndexEll} is a power of $\ell$ by induction on $m$, the case $m=1$ being trivial. Notice that, by Hensel's lemma and since $m \geq 2$, the kernel of $\pi_{m-1}$ is an $\ell$-group (of order $\ell^{\dim \mathcal{G}}$). It follows that $\pi_{m-1}$ induces a surjective map $\mathcal{G}(m;n_1,\ldots,n_t) \to \mathcal{G}(m-1;n_1,\ldots,n_t)$ whose kernel is an $\ell$-group; in particular, the numbers $\frac{\left|\mathcal{G}(m;n_1,\ldots,n_t)\right|}{\left|\mathcal{G}(m-1;n_1,\ldots,n_t)\right|}$ and $\frac{|\mathcal{G}(\mathbb{Z}/\ell^{m}\mathbb{Z})|}{|\mathcal{G}(\mathbb{Z}/\ell^{m-1}\mathbb{Z})|}$ are both powers of $\ell$, and an immediate induction shows that the same is true for \eqref{eq_IndexEll}. 

Choosing $m$ large enough that $H=H[\ell^m]$, it follows from our previous considerations that $\displaystyle [K(H):K(H[\ell])]=\frac{\left[ K(H[\ell^m]):K \right]}{[K(H[\ell]):K]} \uguale \frac{\left[\mathcal{G}(\mathbb{Z}/\ell^m\mathbb{Z}) :\mathcal{G}(m;n_1, \ldots, n_t)\right]}{\left[\mathcal{G}(\mathbb{Z}/\ell\mathbb{Z}) :\mathcal{G}(1;n_1, \ldots, n_t)\right]}$ is a power of $\ell$ (up to bounded constants),
which finishes the proof of the lemma when all the stabilizers $\underline{\mathcal{G}}_I$ are smooth over $\mathbb{Z}_\ell$, and leaves us with only finitely many (not \textit{very good}) primes to consider. To establish the lemma we thus need to show that, for $\ell$ ranging over these finitely many primes and $H$ ranging over the finite subgroups of $A[\ell^\infty]$, the degree $
\left[ K(H) : K(H[\ell]) \right]
$ is within a constant factor of a power of $\ell$. As we are only considering finitely many primes, there are only finitely many subgroups of $A[\ell]$, and therefore we have $\left[K(H[\ell]):K\right] \uguale 1$; hence we just need to show that $[K(H):K]$ is a power of $\ell$ up to a constant factor. Let $\ell^m$ be the exponent of $H$. Since the prime-to-$\ell$ part of $[K(H):K]$ divides the prime-to-$\ell$ part of $[K(A[\ell^m]):K]$, it is enough to show that $|G_{\ell^m}|=|\operatorname{Gal}\left(K(A[\ell^m])/K \right)|$ is a power of $\ell$ up to a bounded constant.
Let $C$ be the least common multiple of the orders of the groups $G_\ell$ for $\ell$ ranging over the finitely many not \textit{very good} primes. Consider the reduction map $\pi: G_{\ell^m} \to G_\ell$, and notice that its kernel is a subgroup of $\ker\left(\operatorname{GL}_{2g}(\mathbb{Z}/\ell^m\mathbb{Z}) \to \operatorname{GL}_{2g}(\mathbb{F}_\ell) \right)$, hence in particular an $\ell$-group; we can then write
$
\frac{\left| G_{\ell^m} \right|}{|\ker \pi|}$ as $\left| \pi\left( G_{\ell^m}  \right)\right|,
$
which by construction is an integer dividing $C$. Since $\left| \ker \pi \right|$ is a power of $\ell$, we see that the prime-to-$\ell$ part of $|G_{\ell^m}|$ is bounded by $C$; this completes the proof in the non-smooth case as well.
\end{proof}

\begin{lemma}\label{lemma_KHOverKJ} Let $K$ be a number field, $A/K$ be an abelian variety, $\ell$ a prime number, and $H$ a finite subgroup of $A[\ell^\infty]$.
We have \[K(H) \cap K(\mu_\ell) \uguale K(H[\ell]) \cap K(\mu_\ell),\] and the degree of $
K(H) \cap K(\mu_{\ell^\infty})$ over $K(H) \cap K(\mu_\ell)
$
is a power of $\ell$.
\end{lemma}

\begin{proof}
Let $m$ be such that $H \subseteq A[\ell^m]$. The Galois group of $K(A[\ell^m])$ over $K(H[\ell]) \cap K(\mu_\ell)$ is generated by $\OtherStabilizer_1:=\operatorname{Gal}\left(K(A[\ell^m]/K(H[\ell]) \right)$
and $N:=\operatorname{Gal}\left(K(A[\ell^m]/K(\mu_\ell) \right)$; 
notice that 
$N=\ker \left( G_{\ell^m} \stackrel{\lambda}{\longrightarrow} \mathbb{F}_\ell^\times \right)$. Let now $\OtherStabilizer_m$ be the Galois group of $K(A[\ell^m])$ over $K(H)$. By lemma \ref{lemma_k0} we see that $[\OtherStabilizer_1:\OtherStabilizer_m]$ is a power of $\ell$ (up to a constant bounded independently of $\ell$), hence
$\displaystyle
\left[ N\OtherStabilizer_1 : N\OtherStabilizer_m  \right] = \frac{|N\OtherStabilizer_{1}/N|}{|N\OtherStabilizer_{m}/N|} = \frac{\left|\lambda(\OtherStabilizer_{1})\right|}{\left|\lambda(\OtherStabilizer_m)\right|}
$ 
is again a power of $\ell$ (up to a constant independent of $\ell$). On the other hand, $\lambda(\OtherStabilizer_{1})$ is a subgroup of $\mathbb{F}_\ell^\times$, hence of order prime to $\ell$: it follows that $\left|\frac{\lambda(\OtherStabilizer_{1})}{\lambda(\OtherStabilizer_m)}\right|\uguale 1$, and therefore $N\OtherStabilizer_{1} \uguale N\OtherStabilizer_m$. Now $N\OtherStabilizer_{1}$ is the Galois group of $K(A[\ell^m])$ over $K(H[\ell]) \cap K(\mu_\ell)$, while $N\OtherStabilizer_m$ is the Galois group of $K(A[\ell^m])$ over $K(H) \cap K(\mu_\ell)$: by Galois theory, this implies $K(H) \cap K(\mu_\ell) \uguale K(H[\ell]) \cap K(\mu_\ell)$ as claimed. The second part is immediate by Galois theory.
\end{proof}


\begin{theorem}{(Theorem \ref{thm_mw})}
Let $K$ be a number field and $A/K$ be an abelian variety. Property $(\mu)_w$ holds for $A$.
\end{theorem}
\begin{proof}
Fix a prime $\ell$ and a finite subgroup $H \subseteq A[\ell^\infty]$: we want to show that we can choose $n$ so as to satisfy inequality \eqref{eq_WeakMu1} (for some constant $C$ only depending on $A/K$). Let $L$ be the intersection $K(H[\ell]) \cap K(\mu_{\ell})$. By lemma \ref{lemma_Core}, we can choose $m \in \left\{0,1\right\}$ so that 
\begin{equation}\label{eq_ChoiceOfm}
\left[ L : K \right] \uguale \left[ K(\mu_{\ell^m}) : K \right],
\end{equation}
and by lemma \ref{lemma_KHOverKJ} we see that there is an integer $j$ such that 
$[K(H) \cap K(\mu_{\ell^\infty}):L] \uguale \ell^j.$
Observe now that
$
\left[ K(H) \cap K(\mu_{\ell^\infty}) : K \right]=\left[K(H)\cap K(\mu_{\ell^\infty}) : L \right][L:K] \uguale \ell^j [L:K],
$
hence by \eqref{eq_ChoiceOfm} we have
$
\left[ K(H) \cap K(\mu_{\ell^\infty}) : K \right] \uguale \ell^j \cdot [K(\mu_{\ell^m}):K].
$
Using the obvious equalities (up to bounded constants) $[K(\mu_{\ell^{j+1}}) : K(\mu_\ell)] \uguale [K(\mu_{\ell^j}):K] \uguale \ell^j$
we deduce
\[
\begin{aligned}
\left[ K(H) \cap K(\mu_{\ell^\infty}) : K \right] & \uguale \ell^j \cdot [K(\mu_{\ell^m}):K] \\ & \uguale [K(\mu_{\ell^{j+m}}):K(\mu_{\ell^m})]\cdot [K(\mu_{\ell^m}):K] \\
& = [K(\mu_{\ell^{j+m}}):K].
\end{aligned}
\]
This shows that, if we take $C$ to be the constant implied in the last formula, for all primes $\ell$ and all finite subgroups $H$ of $A[\ell^\infty]$ inequality \eqref{eq_WeakMu1} can be satisfied by taking $n=m+j$, and therefore property $(\mu)_w$ holds for $A$ as claimed.
\end{proof}

\section{Property $(\mu)_s$}
Let $F$ be any field. We start by considering the representation
\begin{equation}\label{eq_representation}
\begin{array}{cccc}
\rho: & \operatorname{GL}_{2,F} \times \operatorname{GL}_{2,F} \times \operatorname{GL}_{2,F} & \to & \operatorname{GSp}_{8,F} \\
& (a,b,c) & \mapsto & a \otimes b \otimes c,
\end{array}
\end{equation}
where we identify $F^8$ with $F^2 \otimes F^2 \otimes F^2$. We equip $F^8$ with the symplectic form $\psi$ given by $\psi_1 \otimes \psi_2 \otimes \psi_3$, where $\psi_i$ is the standard symplectic form on the $i$-th factor $F^2$: the fact that the action of $\operatorname{GL}_{2,F}$ preserves $\psi_i$ (up to a scalar) implies that the representation $\rho$ does indeed land into $\operatorname{GSp}_{8,F}$. 

\begin{definition}\label{def_M}
We let $M_F$ be the image of this representation: it is an $F$-algebraic group that contains the torus of homotheties.
\end{definition}


\begin{remark}\label{rem_GoodReduction}
Consider the $\mathbb{Z}_\ell$-Zariski closure of $M_{\mathbb{Q}_\ell}$ in $\operatorname{GSp}_{8,\mathbb{Z}_\ell}$, call it $\mathcal{M}_{\mathbb{Z}_\ell}$. By definition, $\mathcal{M}_{\mathbb{Z}_\ell}$ coincides with the $\mathbb{Z}_\ell$-Zariski closure of $M_{\mathbb{Q}} \times_{\mathbb{Q}} \mathbb{Q}_\ell$ in $\operatorname{GSp}_{8,\mathbb{Z}_\ell}$, which is smooth over $\mathbb{Z}_\ell$ for almost all $\ell$ because $M_{\mathbb{Q}}$ extends to a smooth scheme over an open subscheme of $\operatorname{Spec} \mathbb{Z}$. It follows that $\mathcal{M}_{\mathbb{Z}_\ell}$ is smooth over $\mathbb{Z}_\ell$ for almost all $\ell$.
\end{remark}


We think the algebraic group $M_F$ as sitting inside $\mathbb{A}^{64}_F$ (the space of $8 \times 8$ matrices over $F$). It is not hard to find polynomials that belong to the ideal defining $M_F$: by construction $\rho$ factors through $\operatorname{GL}_{2,F} \otimes \operatorname{GL}_{2,F} \otimes \operatorname{GL}_{2,F}$, so 
if we let
$
\left(\begin{matrix} B_{11} & B_{12} \\ B_{21} & B_{22}\end{matrix} \right)
$
be any element in $M_F(\overline{F})$ (where every $B_{ij}$ is a $4 \times 4$ matrix), the construction of the tensor product implies that the four matrices $B_{ij}$ are pairwise linearly dependent (notice that this condition is purely algebraic, being given by the vanishing of sufficiently many determinants). 
Likewise, if we write
$
B_{ij}=\left(\begin{matrix} C_{11} & C_{12} \\ C_{21} & C_{22}\end{matrix} \right),
$
 where each $C_{kl}$ is a $2 \times 2$ matrix, we must again have pairwise linear dependence of the $C_{kl}$, and this (being an algebraic condition) is again true for any point in $M_F(\overline{F})$.
Let now $e_1,e_2$ be the standard basis of $F^2$ and write $e_{ijk}=e_i \otimes e_j \otimes e_k$ (with $i,j,k \in \left\{1,2\right\}$) for the corresponding basis of $F^8$. We order these basis vectors as $e_{111},e_{112}$, $e_{121}$, $e_{122},e_{211},e_{212},e_{221},e_{222}$. The form $\psi$ on $F^8$ is then represented by the matrix
\[
\left(\begin{matrix} 0 & 0 & 0 & 0 & 0 & 0 & 0 & 1 \\ 0 & 0 & 0 & 0 & 0 & 0 & -1 & 0 \\ 0 & 0 & 0 & 0 & 0 & -1 & 0 & 0 \\ 0 & 0 & 0 & 0 & 1 & 0 & 0 & 0 \\ 0 & 0 & 0 & -1 & 0 & 0 & 0 & 0 \\ 0 & 0 & 1 & 0 & 0 &0 &0 &0 \\ 0 & 1 & 0 & 0 & 0 & 0 & 0 & 0 \\ -1 & 0 & 0 & 0 & 0 & 0 & 0 & 0 \end{matrix} \right),
\]
and it is immediate to check that $e_{111}, e_{122}, e_{212}, e_{221}$ span a Lagrangian subspace.

\begin{definition}\label{def_SpecialLagrangianSubspace}
Let $F$ be any field. We let $H$ be the subspace of $F^8\cong \left(F^2\right)^{\otimes 3}$ generated by $e_{111}$, $e_{122}$, $e_{212}$, and $e_{221}$.
\end{definition}

We now determine the stabilizer $\Stabilizer$ of $H$ in $M_F\left( \overline{F} \right)$. In matrix terms, an element $t$ of $\Stabilizer$ can be written as
\[
t=\left(\begin{matrix} 1 & \simb & \simb & 0 & \simb & 0 & 0 & \simb \\
                     0 & \simb & \simb & 0 & \simb & 0 & 0 & \simb \\
                     0 & \simb & \simb & 0 & \simb & 0 & 0 & \simb \\
                     0 & \simb & \simb & 1 & \simb & 0 & 0 & \simb \\
                     0 & \simb & \simb & 0 & \simb & 0 & 0 & \simb \\
                     0 & \simb & \simb & 0 & \simb & 1 & 0 & \simb \\
                     0 & \simb & \simb & 0 & \simb & 0 & 1 & \simb \\
                     0 & \simb & \simb & 0 & \simb & 0 & 0 & \simb \end{matrix} \right),
\]
where each entry $\simb$ is a priori any element of $\overline{F}$. We now use the fact that $\Stabilizer \subseteq M_F(\overline{F})$ to show that $\Stabilizer$ is in fact finite.
Write as before $B_{11}$ (resp.~$B_{12}, B_{21}, B_{22}$) for the top-left (resp.~top-right, bottom-left and bottom-right) block of $t$ of size $4 \times 4$. Since $B_{22}$ is nonzero, linear dependence of $B_{22}$ and $B_{12}$ can be expressed as $B_{12}=\alpha B_{22}$ for a certain $\alpha \in \overline{F}$; however, since $B_{22}$ has some nonzero diagonal coefficients while the corresponding diagonal entries of $B_{12}$ vanish, we must have $\alpha=0$ and $B_{12}=0$. The same argument, applied to $B_{21}$ and $B_{11}$, shows that $B_{21}=0$. 
On the other hand, the blocks $B_{11}$ and $B_{22}$ are both nonzero, so there exists a nonzero $\lambda \in \overline{F}^\times$ such that $B_{22}=\lambda B_{11}$: this leads immediately to
\[
t=\left(\begin{matrix} 1 & 0 & 0 & 0 & 0 & 0 & 0 & 0 \\
                       0 & 1/\lambda & 0 & 0 & 0 & 0 & 0 & 0 \\
                       0 & 0 & 1/\lambda & 0 & 0 & 0 & 0 & 0 \\
                       0 & 0 & 0 & 1 & 0 & 0 & 0 & 0 \\
                       0 & 0 & 0 & 0 & \lambda & 0 & 0 & 0 \\
                       0 & 0 & 0 & 0 & 0 & 1 & 0 & 0 \\
                       0 & 0 & 0 & 0 & 0 & 0 & 1 & 0 \\
                       0 & 0 & 0 & 0 & 0 & 0 & 0 & \lambda \end{matrix} \right).%
\]

We now use the second part of our previous remark, namely the fact that the $2 \times 2$ blocks of $B_{11}$ are linearly dependent as well. Comparing the top-left and bottom-right blocks of $B_{11}$ gives the additional condition $\lambda^2=1$, that is, $\lambda=\pm 1$: thus the stabilizer in $M_F(\overline{F})$ of our Lagrangian subspace $H$ consists of exactly two elements, namely the identity and the operator $\operatorname{diag}(1,-1,-1,1,-1,1,1,-1)$ (at least if $\operatorname{char} F \neq 2$: otherwise we have $-1=1$ and the two coincide). This stabilizer is also clearly finite as an algebraic group, since it has only finitely many points over $\overline{F}$.

Notice that this argument actually shows a little more. Let $\mathcal{M}_{\mathbb{Z}_\ell}$ be the $\mathbb{Z}_\ell$-Zariski closure of $M_{\mathbb{Q}_\ell}$ in $\operatorname{GSp}_{8,\mathbb{Z}_\ell}$. Let furthermore $\mathcal{H}$ be the Lagrangian subspace of $\mathbb{F}_\ell^8 \cong \mathbb{F}_\ell^2 \otimes \mathbb{F}_\ell^2 \otimes \mathbb{F}_\ell^2$ given in definition \ref{def_SpecialLagrangianSubspace} (for the field $\mathbb{F}_\ell$): then the stabilizer of $\mathcal{H}$ in $\mathcal{M}_{\mathbb{Z}_\ell}(\overline{\mathbb{F}_\ell})$ has order at most 2. Indeed, all we have used in the above argument is the linear dependence of certain blocks in the matrix representation of the elements of the stabilizer and the fact that the equation $\lambda^2=1$ admits at most 2 solutions in $\overline{F}$: both properties are also true for the points of $\mathcal{M}_{\mathbb{Z}_\ell}$ with values in any integral $\mathbb{Z}_\ell$-algebra (in particular, $\overline{\mathbb{F}_\ell}$). We record this fact in the following
\begin{proposition}\label{prop_FiniteStabilizer}
Let $\ell$ be a prime, $\mathcal{M}_{\mathbb{Z}_\ell}$ be the $\mathbb{Z}_\ell$-Zariski closure of $M_{\mathbb{Q}_\ell}$ in $\operatorname{GSp}_{8,\mathbb{Z}_\ell}$, and $\mathcal{H}$ be the subspace $H$ of definition \ref{def_SpecialLagrangianSubspace} for the field $\mathbb{F}_\ell$. The stabilizer of $\mathcal{H}$ in $\mathcal{M}_{\mathbb{Z}_\ell}(\overline{\mathbb{F}_\ell})$ consists of at most 2 elements.
\end{proposition}

\subsection{Mumford's examples, and the proof of theorem \ref{thm_ms}}\label{subsect_MumfordExample}
We now recall the construction given by Mumford in \cite{MR0248146}. Suppose we are given the data of a totally real cubic number field $F$ and of a central simple division algebra $D$ over $F$ satisfying:
\begin{enumerate}
\item $\operatorname{Cor}_{F/\mathbb{Q}}(D)=M_8(\mathbb{Q})$;
\item $D \otimes_\mathbb{Q} \mathbb{R} \cong \mathbb{H} \oplus \mathbb{H} \oplus M_2(\mathbb{R})$.
\end{enumerate}

Being a division algebra, $D$ is equipped with a natural involution $x \mapsto \overline{x}$; let $G$ be the $\mathbb{Q}$-algebraic group whose $\mathbb{Q}$-points are given by $\left\{ x \in D^* \bigm\vert x\overline{x}=1 \right\}$. Mumford constructed in \cite{MR0248146} an abelian variety of dimension 4 with trivial endomorphism ring and Hodge group equal to $G$ (in fact, he constructed a Shimura curve parametrizing abelian fourfolds whose Hodge group is contained in $G$, and showed that every sufficiently generic fiber has exactly $G$ as its Hodge group). By specialization, there exists a principally polarized abelian fourfold $A$ defined over a number field $L$ and such that $\operatorname{Hg}(A) \cong G$; since $\operatorname{Hg}(A)$ is as small as it is possible for an abelian fourfold with no additional endomorphisms, the Mumford-Tate conjecture is known to hold for $A$ (cf.~\cite{Moonen95hodgeand}). By theorem \ref{thm_InclusionMT} there is a finite extension $K$ of $L$ such that, if we denote by $G_\ell$ the image of the mod-$\ell$ representation $\abGal{K} \to \operatorname{Aut} A[\ell]$, then we have $G_\ell \subseteq \MT(A)(\mathbb{F}_\ell)$ for all primes $\ell$.
On the other hand, the equality $\operatorname{Cor}_{F/\mathbb{Q}}(D) = M_8(\mathbb{Q})$ implies the existence of a (``norm") map $N:D^* \to \operatorname{GL}_8(\mathbb{Q})$, and Mumford's construction is such that the action of $G(\mathbb{Q})=D^*$ on $V:=H_1(A(\mathbb{C}),\mathbb{Q}) \cong \mathbb{Q}^8$ is given exactly by $N$. 
Furthermore, it is also known that $N$ is a $\mathbb{Q}$-form of the $\mathbb{R}$-representation
$
G(\mathbb{R}) \cong \operatorname{SL}_2(\mathbb{R}) \times \operatorname{SU}_2(\mathbb{R})^2 \to \operatorname{Sp}_8(\mathbb{R})
$
coming from the tensor product of the standard representation of $\operatorname{SL}_2(\mathbb{R})$ by the unique 4-dimensional faithful orthogonal representation $\operatorname{SU}_2(\mathbb{R})^2 \to \operatorname{SO}_4(\mathbb{R})$. In particular, by extension of scalars to $\mathbb{C}$ we see that the action of $G(\mathbb{C}) \cong \operatorname{SL}_2(\mathbb{C})^3$ on $V_\mathbb{C}$ is given by the representation $\rho$ of the previous paragraph (restricted to $\operatorname{SL}_2(\mathbb{C})^3$).


\begin{lemma}\label{lemma_Inclusion}
Let $\ell$ be a prime such that $G \times_{\mathbb{Q}} \mathbb{Q}_\ell$ is split. Then (up to choosing a suitable identification $T_\ell(A) \otimes \mathbb{Q}_\ell \cong \mathbb{Q}_\ell^8$) we have $\MT(A)\times_{\mathbb{Z}} \mathbb{Q}_\ell = M\times_{\mathbb{Q}} \mathbb{Q}_\ell$, where $M=M_{\mathbb{Q}}$ is the algebraic group of definition \ref{def_M} for the field $\mathbb{Q}$.
\end{lemma}
\begin{proof}
The morphism $G \to \operatorname{Sp}_{8,\mathbb{Q}}$ is given by the norm map, and if $G \times_{\mathbb{Q}} \mathbb{Q}_\ell$ is split (hence isomorphic to $\operatorname{SL}_{2,\mathbb{Q}_\ell}^3$) the norm map is exactly
\[
\begin{array}{cccc}
\rho : & \operatorname{SL}_{2,\mathbb{Q}_\ell}^3 & \to & \operatorname{Sp}_{8,\mathbb{Q}_\ell} \\
& (a,b,c) & \mapsto & a \otimes b \otimes c;
\end{array}
\]
it follows that $M \times_{\mathbb{Q}} \mathbb{Q}_\ell$ contains $\Hg(A) \times_{\mathbb{Q}} \mathbb{Q}_\ell$ (as algebraic groups). On the other hand, $\MT(A)$ is the almost-direct product of $\Hg(A)$ by the homotheties torus $\mathbb{G}_m$, and we know that $M$ also contains $\mathbb{G}_m$. This proves that we have $\MT(A) \times \mathbb{Q}_\ell \subseteq M \times \mathbb{Q}_\ell$, and since the two groups have the same dimension the inclusion must be an equality.
\end{proof}

Extend now $M$ and $G$ to group schemes over $\mathbb{Z}$ by taking their $\mathbb{Z}$-Zariski closure in their respective ambient spaces; there is an open subscheme $\operatorname{Spec} \mathbb{Z}\left[ \frac{1}{S} \right]$ of $\operatorname{Spec} \mathbb{Z}$ over which $M, \MT(A)$ and $G$ are all smooth. Consider the family $\mathcal{F}$ of primes $\ell$ unramified in $K$, such that $G$ splits over $\mathbb{Q}_\ell$, and which do not divide $S$. We claim that $\mathcal{F}$ is infinite. Indeed, for $G$ to be split over $\mathbb{Q}_\ell$ it is enough that the root datum of $G$ be unramified at $\ell$ and that the Frobenius at $\ell$ act trivially on it, which -- by Chebotarev's theorem -- is the case for a positive-density set of primes (the action of $\abGal{\mathbb{Q}}$ on the root datum of $G$ factors through a finite quotient): it is then clear that $\mathcal{F}$ is infinite, because only finitely many primes divide $S$ or the discriminant of $K$.
Pick now any $\ell$ in $\mathcal{F}$ and let $\mathcal{M}=M \times_{\mathbb{Z}} \mathbb{Z}_\ell$. The definition of $\mathcal{F}$ implies that $\mathcal{M}$ is a smooth $\mathbb{Z}_\ell$-model of $M \times_{\mathbb{Z}} \mathbb{Q}_\ell=M_{\mathbb{Q}_\ell}$, and by lemma \ref{lemma_Inclusion} we have $\MT(A)\times_{\mathbb{Z}} \mathbb{Z}_\ell=\mathcal{M}$, because both groups can be obtained as the $\mathbb{Z}_\ell$-Zariski closure of the same generic fiber. In particular, we see that $G_\ell$ is contained in $\mathcal{M}(\mathbb{F}_\ell)=\MT(A)(\mathbb{F}_\ell)$.
Take now $\mathcal{H} \subseteq A[\ell]$ to be the Lagrangian subspace of definition \ref{def_SpecialLagrangianSubspace} (for the field $\mathbb{F}_\ell$). The field $K(\mathcal{H})$ is clearly contained in $K(A[\ell])$, so in order to describe $K(\mathcal{H})$ it suffices to describe $\operatorname{Gal}\left(K(A[\ell])/K(\mathcal{H}) \right)$, that is, the stabilizer of $\mathcal{H}$ in $G_\ell$; as $G_\ell$ is contained in $\mathcal{M}(\mathbb{F}_\ell)$, this stabilizer is certainly contained in the stabilizer of $\mathcal{H}$ in $\mathcal{M}(\mathbb{F}_\ell)$, which in turn consists of at most two elements by proposition \ref{prop_FiniteStabilizer}. We have thus proved that the index $[K(A[\ell]):K(\mathcal{H})]$ is at most $2$, and since $K(\mu_\ell)$ is contained in $K(A[\ell])$ by the properties of the Weil pairing (recall that $A$ is principally polarized) we have
\[
\left[K(\mathcal{H}) \cap K(\mu_{\ell^\infty}) : K \right] \geq \frac{1}{2} \left[K(A[\ell]) \cap K(\mu_{\ell^\infty}) : K \right] \geq \frac{1}{2} \left[ K(\mu_\ell):K \right] = \frac{\ell-1}{2},
\]
where the last equality follows from the fact that $\ell$ is unramified in $K$. We then see that property $(\mu)_s$ does not hold for Mumford's example: indeed, $\mathcal{H}$ is Lagrangian, hence we have $m_1(\mathcal{H})=0$; but if property $(\mu)_s$ held for $A/K$, then (for some $C$) the inequality
\[
\frac{\ell-1}{2} \leq [K(\mathcal{H}) \cap K(\mu_{\ell^\infty}) : K ] \leq C \left[ K(\mu_{\ell^{m_1(\mathcal{H})}}) : K \right] = C
\]
would be satisfied by all the primes in our infinite family $\mathcal{F}$, and this is clearly absurd. This establishes theorem \ref{thm_ms}.


\medskip
\noindent\textbf{Acknowledgments.} I am grateful to Nicolas Ratazzi for attracting my interest to the problem considered in this paper. I thank Antonella Perucca for useful discussions and for pointing out Example \ref{ex_SelfProducts}, and the anonymous referee for suggesting that theorem \ref{thm_mw} could be made independent of the truth of the Mumford-Tate conjecture. The author gratefully acknowledges financial support from the Fondation Mathématique Jacques Hadamard.

\bibliography{Biblio}{}
\bibliographystyle{plain}

\break

\end{document}